\documentclass{article}
\usepackage[english]{babel}
\usepackage{amsmath,amssymb,latexsym,alltt}

\usepackage[colorlinks=true,linkcolor=blue,citecolor=blue,urlcolor=black,pdfpagelabels=false]{hyperref}

\newcommand{\md}{\mathrm{d}}
\newcommand{\pFq}[5]{\ensuremath{{}_{#1}F_{#2} \left[ \genfrac{}{}{0pt}{}{#3}{#4} \middle| {#5} \right]}}
\renewcommand{\Re}{\operatorname{Re}}

\newcommand{\email}[1]{{\textit{Email:} \texttt{#1}}}

\usepackage{amsthm,thmtools}
\theoremstyle{plain}
  \declaretheorem[numberwithin=section]{theorem}

\theoremstyle{definition}
  
  \declaretheorem[numberlike=theorem]{example}
  \declaretheorem[numberlike=theorem]{remark}

\begin{document}

\title{Interpolated sequences and critical $L$-values of modular forms}

\author{Robert Osburn
\thanks{School of Mathematics and Statistics, University College Dublin, Belfield, Dublin 4, Ireland, \email{robert.osburn@ucd.ie}}
\and Armin Straub
\thanks{Department of Mathematics and Statistics, University of South Alabama, 411 University Blvd N, Mobile, AL 36688, United States, \email{straub@southalabama.edu}}}
 
\date{September 19, 2018}

\maketitle

\begin{abstract}
  Recently, Zagier expressed an interpolated version of the Ap{\'e}ry numbers
  for $\zeta(3)$ in terms of a critical $L$-value of a modular form of weight
  4.  We extend this evaluation in two directions. We first prove that
  interpolations of Zagier's six sporadic sequences are essentially critical
  $L$-values of modular forms of weight 3. We then establish an infinite
  family of evaluations between interpolations of leading coefficients of
  Brown's cellular integrals and critical $L$-values of modular forms of odd
  weight.
\end{abstract}

\section{Introduction}

For $x \in \mathbb{C}$, consider the absolutely convergent series
\begin{equation}
  A (x) = \sum_{k = 0}^{\infty} \binom{x}{k}^2 \binom{x + k}{k}^2.
  \label{eq:apery3:x}
\end{equation}
If $x=n \in \mathbb{Z}_{\geq 0}$, this series terminates at $k=n$ and agrees with the well-known Ap{\'e}ry numbers $A(n)$ for $\zeta(3)$ \cite{apery}, \cite{alf}. Let 
\begin{equation*}
f(\tau) = \sum_{n \geq 1} a_n q^n \in S_k(\Gamma_1(N)), \quad q = e^{2 \pi i \tau},
\end{equation*}
be a cusp form of weight $k$ and level $N$, and
\begin{equation*}
L(f,s) := \sum_{n \geq 1} \frac{a_n}{n^s}
\end{equation*}
be the $L$-function for $f$ defined for $\Re{s}$ large. Following Deligne, we say that $L(f,j)$ is a \emph{critical $L$-value} if $j \in \{1,2,\ldots,k-1 \}$. For a beautiful exposition concerning the importance of these numbers, see \cite{kz}. Zagier \cite[(44)]{zagier-de} recently showed that the interpolated Ap{\'e}ry numbers \eqref{eq:apery3:x} are related to the critical $L$-value of a modular form of weight 4. Specifically, he proved the following intriguing identity:

\begin{equation}
  A (-\tfrac{1}{2}) = \frac{16}{\pi^2} L (f, 2),
  \label{eq:apery3:L}
\end{equation}
where
\begin{equation} \label{eq:Aweight4}
  f(\tau) = \eta (2 \tau)^4 \eta (4 \tau)^4 = \sum_{n \geq 1}
   \alpha_n q^n 
\end{equation}
is the unique normalized Hecke eigenform in $S_4 (\Gamma_0 (8))$ and $\eta(\tau)$ is the Dedekind eta function.
We note that, expressing the left-hand side as a hypergeometric series, the identity \eqref{eq:apery3:L} was previously established by Rogers, Wan and Zucker \cite{rwz-ell}.
The evaluation \eqref{eq:apery3:L} can be seen as a continuous counterpart to the congruence
\begin{equation}
  A (\tfrac{p - 1}{2}) \equiv \alpha_p \quad (\operatorname{mod} p),
  \label{eq:apery3:mod}
\end{equation}
which holds for primes $p > 2$ and was established by Beukers
\cite{beukers-apery87}, who further conjectured that the congruence
\eqref{eq:apery3:mod} actually holds modulo $p^2$. This supercongruence
was later proven by Ahlgren and Ono \cite{ahlgren-ono-apery} using Gaussian hypergeometric
series.

Zagier indicates that Golyshev predicted an evaluation of the form
\eqref{eq:apery3:L} based on motivic considerations and the connection of the
Ap\'ery numbers with the double covering of a related family of K3 surfaces.
Here, we do not touch on these geometric considerations (see
\cite[Section~7]{zagier-de} for further details), but only note that
Golyshev's prediction further relies on the Tate conjecture, which remains
open in the required generality. Identity \eqref{eq:apery3:L}, and similar
ones to be explored in this paper, might therefore serve as evidence
supporting the motivic philosophy and the Tate conjecture.

The goal of this paper is to extend Zagier's evaluation \eqref{eq:apery3:L} in
two directions. Firstly, in Section~\ref{sec:apery2}, we consider the six sporadic sequences
that Zagier \cite{zagier4} obtained as integral solutions to Ap\'ery-like
second order recurrences. Based on numerical experiments, we observe that each
of these sequences $C_{*}(n)$ appears to satisfy congruences like
\eqref{eq:apery3:mod} connecting them with the Fourier coefficients of a modular
form $f_{*}(\tau)$ of weight $3$. For three of these sequences these congruences
were shown by Stienstra and Beukers \cite{sb-picardfuchs}, while the other
three congruences do not appear to have been recorded before.
We prove two of these new cases, one using a general
result of Verrill \cite{verrill-cong} and the other via $p$-adic analysis and comparison with another case. Our
main objective is to show that in each case there is a version
of Zagier's evaluation \eqref{eq:apery3:L}. For $x \in \mathbb{C}$, there is a natural
interpolation $C_{*}(x)$ of each sequence and the value $C_{*}(- 1 / 2)$ can in five of
the six cases be expressed as $\frac{\alpha}{\pi^2} L (f_{*}, 2)$ for $\alpha \in \{ 6, 8, 12, 16 \}$. In the remaining case, $C_{*}(x)$ has a pole at
$x = - 1 / 2$. Remarkably, the residue of that pole equals $\frac{6}{\pi^2} L
(f_{*}, 1)$.

Secondly, Brown \cite{brown-apery} recently introduced cellular integrals
generalizing the linear forms used in Ap\'ery's proof of the irrationality
of $\zeta (3)$ as well as many other constructions related to the
irrationality of zeta values. These are linear forms in multiple zeta values
and their leading coefficients $A_{\sigma}(n)$ are generalizations of the Ap\'ery
numbers. McCarthy and the authors \cite{mos-brown} proved that, for a
certain infinite family of these cellular integrals, the leading coefficients
$A_{\sigma}(n)$ satisfy congruences like \eqref{eq:apery3:mod} with Fourier
coefficients of modular forms $f_k (\tau)$ of odd weight $k \geq 3$. In
Section~\ref{sec:cellular}, we review these facts and prove an analogue of
Zagier's evaluation \eqref{eq:apery3:L} for all of these sequences. Finally, in Section~\ref{sec:outlook}, we conclude with several directions for future study.

\section{Zagier's sporadic sequences}\label{sec:apery2}

\subsection{The congruences and $L$-value relations}

In addition to $A(n)$, Ap\'ery \cite{apery} introduced a second sequence which allowed him to
reprove the irrationality of $\zeta (2)$. This sequence is the solution of
the three-term recursion, for $(a, b, c) = (11, 3, - 1)$,
\begin{equation}
  (n + 1)^2 u_{n + 1} = (a n^2 + a n + b) u_n - c n^2 u_{n - 1},
  \label{eq:rec2-abc}
\end{equation}
with initial conditions $u_{- 1} = 0$, $u_0 = 1$. Inspired by Beukers
\cite{beukers-dwork}, Zagier \cite{zagier4} conducted a systematic search
for parameters $(a, b, c)$ which similarly result in integer solutions to the
recurrence \eqref{eq:rec2-abc}. After normalizing, and apart from degenerate cases, he
discovered four hypergeometric, four Legendrian and six sporadic
solutions. It remains an open question whether this list is complete. The six
sporadic solutions are listed in Table~\ref{tbl:sporadic2:i}. As in \cite{zagier4}, we use the labels  
$\boldsymbol{A}$-$\boldsymbol{F}$ and index the sequences accordingly.

For each of these sequences, a binomial sum representation is known. For
instance, if $(a, b, c) = (11, 3, - 1)$, then
\begin{equation}
  C_{\boldsymbol{D}} (n) = \sum_{k = 0}^n \binom{n}{k}^2 \binom{n + k}{k}.
  \label{eq:apery2}
\end{equation}
Following Zagier's approach for \eqref{eq:apery3:x}, we obtain an
interpolation of a sporadic sequence by replacing the integer $n$ in
the binomial representation with a complex number $x$ and extending the sum to all nonnegative integers $k$. Note that some care is needed for sequence
$\boldsymbol{C}$ (see Example \ref{eq:CC:interpol} in Section~\ref{sec:interpol}). The resulting interpolations are recorded in Table~\ref{tbl:sporadic2:i}. We note that this
construction depends on the binomial sum which is neither unique
nor easily obtained from the recursion \eqref{eq:rec2-abc}. The fact that we
can relate the value of these interpolations at $x=-1/2$ to critical $L$-values, as in Zagier's
evaluation \eqref{eq:apery3:L}, indicates that our choices are natural. We offer some more comments on
these interpolations in Section~\ref{sec:interpol}.

\begin{table}[h]
\begin{equation*}
\small{\arraycolsep=2pt\def\arraystretch{2}\begin{array}{|l|l|l|l|l|l|}
\hline 
* & C_{*}(n) & C_{*}(x) \\ \hline
\boldsymbol{A} & \sum\limits_{k = 0}^n \binom{n}{k}^3 & \sum\limits_{k \geq 0} \binom{x}{k}^3 \\ \hline 
\boldsymbol{B} & \sum\limits_{k = 0}^{\lfloor n/3 \rfloor} (- 1)^k 3^{n - 3 k} \binom{n}{3 k} \frac{(3 k) !}{k!^3} & \sum\limits_{k \geq 0} (- 1)^k 3^{x - 3 k} \binom{x}{3 k} \frac{(3 k) !}{k!^3} \\ \hline
\boldsymbol{C} & \sum\limits_{k = 0}^n \binom{n}{k}^2 \binom{2k}{k} & \Re \pFq{3}{2}{- x, - x, 1 / 2}{1, 1}{4} \\ \hline
\boldsymbol{D} &\sum\limits_{k = 0}^n \binom{n}{k}^2 \binom{n + k}{k}  & \sum\limits_{k \geq 0} \binom{x}{k}^2 \binom{x + k}{k} \\ \hline
\boldsymbol{E} &\sum\limits_{k = 0}^n \binom{n}{k} \binom{2 k}{k}\binom{2 (n - k)}{n - k} & \sum\limits_{k \geq 0} \binom{x}{k} \binom{2 k}{k}\binom{2 (x - k)}{x - k} \\ \hline
 \boldsymbol{F} & \sum\limits_{k = 0}^n (- 1)^k 8^{n - k} \binom{n}{k} C_{\boldsymbol{A}}(k) & \sum\limits_{k \geq 0} (- 1)^k 8^{x - k} \binom{x}{k} C_{\boldsymbol{A}} (k) \\ \hline
 \end{array}}
  \end{equation*}
  \caption{\label{tbl:sporadic2:i} Zagier's six sporadic sequences \cite{zagier4} and their interpolations}
\end{table}

\begin{table}[h]
  \begin{equation*}
    \small{\arraycolsep=2pt\def\arraystretch{2}\begin{array}{|l|l|l|l|l|l|}
       \hline
        * & f_{*}(\tau) & N_{*} & L (f_*, 2) & \alpha_{*} \\
       \hline
       \boldsymbol{A} & \frac{\eta (4
       \tau)^5 \eta (8 \tau)^5}{\eta (2 \tau)^2 \eta (16 \tau)^2} & 32 &
       \frac{\Gamma^2 \left(\frac{1}{8} \right) \Gamma^2 \left(\frac{3}{8}
       \right)}{64 \sqrt{2} \pi} & 8\\
       \hline
       \boldsymbol{B} & \eta (4 \tau)^6 & 16 & \frac{\Gamma^4 \left(\frac{1}{4} \right)}{64 \pi} & 8\\
       \hline
       \boldsymbol{C} & \eta (2 \tau)^3 \eta (6 \tau)^3 & 12 &
       \frac{\Gamma^6 \left(\frac{1}{3} \right)}{2^{17 / 3} \pi^2} & 12\\
       \hline
       \boldsymbol{D} & \eta (4 \tau)^6 & 16 & \frac{\Gamma^4 \left(\frac{1}{4} \right)}{64
       \pi} & 16\\
       \hline
       \boldsymbol{E} & \eta (\tau)^2 \eta (2 \tau) \eta (4 \tau)
       \eta (8 \tau)^2 & 8 & \frac{\Gamma^2 \left(\frac{1}{8} \right)
       \Gamma^2 \left(\frac{3}{8} \right)}{192 \pi} & 6\\
       \hline
       \boldsymbol{F} & q - 2 q^2 + 3 q^3 + \ldots & 24 & \frac{\Gamma
       \left(\frac{1}{24} \right) \Gamma \left(\frac{5}{24} \right) \Gamma
       \left(\frac{7}{24} \right) \Gamma \left(\frac{11}{24} \right)}{96
       \sqrt{6} \pi} & 6\\
       \hline
     \end{array}}
  \end{equation*}
  \caption{\label{tbl:sporadic2} The weight 3, level $N_{*}$ newforms $f_{*}$ with their $L$-values}
\end{table}

Our first main result is the following.

\begin{theorem}
  \label{thm:C2:mod}Let $C_{*}(n)$ be the sporadic sequence in Table~\ref{tbl:sporadic2:i} and $f_{*}(\tau) = \sum_{n \geq 1} \gamma_{n,*} q^n$ be the weight 3, level $N_{*}$ newform listed in Table~\ref{tbl:sporadic2} where the label $*$ is $\boldsymbol{A}$, $\boldsymbol{B}$, $\boldsymbol{C}$, $\boldsymbol{D}$ or $\boldsymbol{E}$. Then, for all primes $p > 2$,
  \begin{equation}
    C_{*}(\tfrac{p - 1}{2}) \equiv \gamma_{p, *} \quad (\operatorname{mod} p) .
    \label{eq:C2:mod}
  \end{equation}
\end{theorem}

We note that the congruences \eqref{eq:C2:mod} hold modulo $p^2$ only for sequence $\boldsymbol{D}$ \cite{ahlgren-sc}. Theorem~\ref{thm:C2:mod} is known to be true for sequences $\boldsymbol{C}$ and
$\boldsymbol{D}$ by work of Stienstra and Beukers \cite{sb-picardfuchs}, and
we show in Section~\ref{sec:mod} that the congruences for sequence
$\boldsymbol{A}$ can be deduced from their work. The other three cases,
including $\boldsymbol{F}$, appear not to have been
considered before. We also show in Section~\ref{sec:mod} that a general result
of Verrill \cite{verrill-cong} can be used to prove the modular congruences
of Theorem~\ref{thm:C2:mod} for sequences $\boldsymbol{C}$ and $\boldsymbol{E}$.
As she points out with sequence $\boldsymbol{A}$, the same
approach does not apply in the other cases. Verrill indicates that the modular
congruences for sequence $\boldsymbol{A}$ can be explained by
Atkin-Swinnerton-Dyer congruences \cite[Section~5.8]{ll-asd}. We expect that
similar ideas can be applied to the case $\boldsymbol{F}$, for which we have numerically observed that Theorem~\ref{thm:C2:mod} holds as well.

For our second main result, we have the following analogues of Zagier's evaluation \eqref{eq:apery3:L}.

\begin{theorem}
  \label{thm:C2:L}Let $C_{*}(x)$ be the interpolated sequence in Table~\ref{tbl:sporadic2:i} and $f_{*}(\tau)$ be the weight 3, level $N_{*}$ newform listed in Table~\ref{tbl:sporadic2} where the label $*$ is $\boldsymbol{A}$, $\boldsymbol{B}$, $\boldsymbol{C}$, $\boldsymbol{D}$ or $\boldsymbol{F}$. Then 
  \begin{equation}
    C_{*} (-\tfrac{1}{2}) = \frac{\alpha_{*}}{\pi^2} L (f_{*}, 2) .
    \label{eq:C2:L}
  \end{equation}
 For sequence $\boldsymbol{E}$,
  \begin{equation}
    \underset{x = - 1 / 2}{\operatorname{res}} C_{\boldsymbol{E}}(x) = \frac{6}{\pi^2} L (f_{\boldsymbol{E}}, 1) .
    \label{eq:CE:L}
  \end{equation}
\end{theorem}

We prove Theorem~\ref{thm:C2:L} in Section~\ref{sec:L}.
The proof for sequence $\boldsymbol{F}$, using a modular parametrization from
\cite{ctyz-clausen}, is due to Wadim Zudilin.
We note from Table~\ref{tbl:sporadic2} that $\alpha_{*}$ divides $N_{*}$ in
all cases except $\boldsymbol{E}$. It is natural to wonder if a uniform
explanation can be given for this observation.

Finally, we note that for sequence $\boldsymbol{E}$, \eqref{eq:CE:L} can be
written in terms of $L (f_{\boldsymbol{E}}, 2)$ by virtue of the relation
\begin{equation}
  L (f_{\boldsymbol{E}}, 1) = \frac{\sqrt{2}}{\pi} L (f_{\boldsymbol{E}}, 2) . \label{eq:LE:1:2}
\end{equation}
This is an instance of a general principle, briefly discussed at the end of
Section~\ref{sec:cellular}, which implies that the normalized critical
$L$-values of $f_{\boldsymbol{E}}(\tau)$ are algebraic multiples of each other.

\subsection{Proof of Theorem \ref{thm:C2:mod}}\label{sec:mod}

Zagier showed that each of the sporadic sequences $C_{*}(n)$ in
Table~\ref{tbl:sporadic2:i} has a modular parametrization, that is, there exists
a modular function $x (\tau)$ such that
\begin{equation}
  y (\tau) := \sum_{n = 0}^{\infty} C (n) x (\tau)^n \label{eq:C:modxy}
\end{equation}
is a modular form of weight $1$. In cases $\boldsymbol{C}$ and $\boldsymbol{E}$,
these are connected to the corresponding modular form $f_{*}(\tau)$ in
Table~\ref{tbl:sporadic2} in such a way that we can apply a general result of
Verrill \cite{verrill-cong} to prove the modular congruences claimed in
Theorem~\ref{thm:C2:mod}. This general result is an extension of Beukers' proof
\cite{beukers-apery87}, which we revisit in
Example~\ref{eg:verrill:beukers}, of the congruences \eqref{eq:apery3:mod} for
the Ap\'ery numbers.

\begin{theorem}[{\cite[Theorem~1.1]{verrill-cong}}]
  \label{thm:verrill}Let $y (\tau)$ be a modular form of weight $k$ and $x
  (\tau)$ a modular function of level $N$, and define $C (n)$ by
  \eqref{eq:C:modxy}. Suppose that, for some integers $M$ and $a_d$,
  \begin{equation*}
    y \frac{q}{x} \frac{\md x}{\md q} = \sum_{d|M} a_d f (d \tau),
  \end{equation*}
  where $f (\tau) = \sum \gamma_n q^n$ is a weight $k+2$, level $N$ Hecke eigenform with character $\chi$. Then,
  \begin{equation*}
    C (m p^r) - \gamma_p C (m p^{r - 1}) + \chi (p) p^{k + 1} C (m p^{r - 2})
     \equiv 0 \quad (\operatorname{mod} p^r),
  \end{equation*}
  for any prime $p \nmid N M$ and integers $m, r$. In particular, if $C (1) =
  1$, then
  \begin{equation*}
    C (p) \equiv \gamma_p \quad (\operatorname{mod} p) .
  \end{equation*}
\end{theorem}

In the next example, we apply Theorem~\ref{thm:verrill} to deduce the congruences
\eqref{eq:apery3:mod} for the Ap\'ery numbers $A(n)$ (see also \cite[Section~2.1]{verrill-cong}).

\begin{example}
  \label{eg:verrill:beukers}As shown in \cite{beukers-apery87}, the
  Ap\'ery numbers $A (n)$ have
  the modular parametrization \eqref{eq:C:modxy} with
  \begin{equation*}
    x (\tau) = \frac{\eta (\tau)^{12} \eta (6 \tau)^{12}}{\eta (2 \tau)^{12}
     \eta (3 \tau)^{12}}, \quad y (\tau) = \frac{\eta (2 \tau)^7 \eta (3
     \tau)^7}{\eta (\tau)^5 \eta (6 \tau)^5} .
  \end{equation*}
  Observe that, defining $\tilde{x} (\tau)$ and $\tilde{y} (\tau)$ by
  \begin{equation*}
    \tilde{x} (\tau)^2 = x (2 \tau), \quad \tilde{y} (\tau) = \tilde{x}
     (\tau) y (2 \tau),
  \end{equation*}
  we have from \eqref{eq:C:modxy} that
  \begin{equation*}
    \tilde{y} (\tau) = \sum_{\substack{
       n = 0\\
       \text{$n$ odd}
     }}^{\infty} A(\tfrac{n - 1}{2}) \tilde{x}
     (\tau)^n .
  \end{equation*}
  It then follows from 
  \begin{equation*}
    \tilde{y} \frac{q}{\tilde{x}} \frac{\md \tilde{x}}{\md q} = f
     (\tau) - 9 f (3 \tau),
  \end{equation*}
  where $f (\tau)$ is given by \eqref{eq:Aweight4}, and Theorem~\ref{thm:verrill} that the congruences
  \eqref{eq:apery3:mod} hold for primes $p > 3$. This is the proof given in
  \cite{beukers-apery87}, which is generalized to Theorem~\ref{thm:verrill}
  in \cite{verrill-cong}.
\end{example}

We now proceed with the proof of Theorem \ref{thm:C2:mod}.

\begin{proof}[Proof of Theorem \ref{thm:C2:mod}]
  We first recall that the cases $\boldsymbol{C}$ and $\boldsymbol{D}$ were already proved in
  \cite{sb-picardfuchs}. To alternatively deduce case $\boldsymbol{C}$ from
  Theorem~\ref{thm:verrill}, we note that $C_{\boldsymbol{C}} (n)$ has the modular parametrization
  \eqref{eq:C:modxy} with (see \cite{zagier4})
  \begin{equation*}
    x (\tau) = \frac{\eta (\tau)^4 \eta (6 \tau)^8}{\eta (2 \tau)^8 \eta (3
     \tau)^4}, \quad y (\tau) = \frac{\eta (2 \tau)^6 \eta (3 \tau)}{\eta
     (\tau)^3 \eta (6 \tau)^2} .
  \end{equation*}
  Defining $\tilde{x} (\tau)$ and $\tilde{y} (\tau)$ from $x (\tau)$ and $y
  (\tau)$ as in Example~\ref{eg:verrill:beukers}, it then follows from
  \begin{equation*}
    \tilde{y} \frac{q}{\tilde{x}} \frac{\md \tilde{x}}{\md q} =
     f_{\boldsymbol{C}} (\tau)
  \end{equation*}
  and Theorem~\ref{thm:verrill} that the congruences \eqref{eq:C2:mod} hold
  for sequence $\boldsymbol{C}$.
  
  Similarly, it is shown in \cite{zagier4} that $C_{\boldsymbol{E}}
  (n)$ has the modular parametrization \eqref{eq:C:modxy} with
  \begin{equation*}
    x (\tau) = \frac{\eta (\tau)^4 \eta (4 \tau)^2 \eta (8 \tau)^4}{\eta (2
     \tau)^{10}}, \quad y (\tau) = \frac{\eta (2 \tau)^{10}}{\eta (\tau)^4
     \eta (4 \tau)^4} .
  \end{equation*}
  Again, defining $\tilde{x} (\tau)$ and $\tilde{y} (\tau)$ as in
  Example~\ref{eg:verrill:beukers}, it follows from
  \begin{equation*}
    \tilde{y} \frac{q}{\tilde{x}} \frac{\md \tilde{x}}{\md q} =
     f_{\boldsymbol{E}} (\tau) + 2 f_{\boldsymbol{E}} (2 \tau)
  \end{equation*}
  and Theorem~\ref{thm:verrill} that the congruences \eqref{eq:C2:mod} hold
  for sequence $\boldsymbol{E}$.
We now claim that 
 \begin{equation}
    C_{\boldsymbol{A}} (\tfrac{p - 1}{2}) \equiv \gamma_{p, \boldsymbol{A}} \quad (\operatorname{mod} p),
  \label{eq:CA:mod1}
  \end{equation}
where $\gamma_{p,\boldsymbol{A}}$ is the $p$th Fourier coefficient of $f_{\boldsymbol{A}}(\tau)$. To see this, note that (see \cite{sb-picardfuchs}) for primes $p>2$, 
\begin{equation}
    (- 1)^{(p - 1) / 2} C_{\boldsymbol{A}} (\tfrac{p - 1}{2}) \equiv
    \gamma_{p, \boldsymbol{E}} \quad (\operatorname{mod} p), \label{eq:CA:mod:sb}
  \end{equation}
  (this congruence is recorded in \cite[(4.55)]{verrill-cong} with the sign
  missing) where $\gamma_{p, \boldsymbol{E}}$ is the $p$th Fourier coefficient of $f_{\boldsymbol{E}}(\tau)$.  Now, observe the relation
 \begin{equation*}
    (- 1)^{(n - 1) / 2} \gamma_{n, \boldsymbol{A}} = \gamma_{n, \boldsymbol{E}} + 2 \gamma_{n/2, \boldsymbol{E}}
  \end{equation*}
for all integers $n \geq 1$. In particular, for odd $n$, $\gamma_{n, \boldsymbol{E}} = (- 1)^{(n - 1) / 2} \gamma_{n, \boldsymbol{A}}$. Thus, \eqref{eq:CA:mod:sb} is equivalent to \eqref{eq:CA:mod1}. We note that \eqref{eq:CA:mod:sb} provides a quick alternative proof of the congruences for sequence $\boldsymbol{E}$ by showing that
  \begin{equation*}
    C_{\boldsymbol{E}} (\tfrac{p-1}{2}) \equiv (- 1)^{(p - 1) /
     2} C_{\boldsymbol{A}} (\tfrac{p-1}{2}) \quad (\operatorname{mod} p).
  \end{equation*}
  This congruence can be deduced directly from the
  binomial sums recorded in Table~\ref{tbl:sporadic2:i} and the fact that the
  congruences hold termwise.

  Finally, let us prove the congruences for sequence $\boldsymbol{B}$.
  Expressing the defining binomial sum hypergeometrically, we have
  \begin{equation*}
    C_{\boldsymbol{B}} (\tfrac{p - 1}{2} ) = 3^{(p - 1) / 2} \pFq{3}{2}{\tfrac{1 - p}{6}, \tfrac{3 - p}{6}, \tfrac{5 - p}{6}}{1, 1}{1} .
  \end{equation*}
  Because the hypergeometric series is a finite sum (one of the top parameters
  is a negative integer), it follows that
  \begin{equation*}
    C_{\boldsymbol{B}} (\tfrac{p - 1}{2} ) \equiv 3^{(p - 1) / 2}
     \pFq{3}{2}{\tfrac{1 - p}{6}, \tfrac{3 - p}{6}, \tfrac{5 - p}{6}}{1 - \tfrac{p}{6}, 1 - \tfrac{p}{3}}{1} \quad (\operatorname{mod} p) .
  \end{equation*}
  By specializing Watson's identity (see, for instance,
  \cite[Theorem~3.5.5(i)]{aar}), we find that this hypergeometric sum has
  the closed form evaluation
  \begin{equation}
    \pFq{3}{2}{\tfrac{1 - p}{6}, \tfrac{3 - p}{6}, \tfrac{5 - p}{6}}{1 - \tfrac{p}{6}, 1 - \tfrac{p}{3}}{1} = \left(\frac{\Gamma (\tfrac{1}{2} )
    \Gamma (1 - \tfrac{p}{6} )}{\Gamma (\tfrac{7 - p}{12}
    ) \Gamma (\tfrac{11 - p}{12} )} \right)^2 .
    \label{eq:CB:gamma}
  \end{equation}
  If $p \equiv 3 \pmod{4}$, then $p \equiv 7, 11
  \pmod{12}$ and we see that the right-hand side of
  \eqref{eq:CB:gamma} is zero, so that $C_{\boldsymbol{B}} (\tfrac{p -
  1}{2} )$ vanishes modulo $p$. Suppose that $p \equiv 1 \pmod{4}$. With some care, we are able to write
  \begin{equation*}
    \frac{\Gamma (\tfrac{1}{2} ) \Gamma (1 - \tfrac{p}{6}
     )}{\Gamma (\tfrac{7 - p}{12} ) \Gamma (\tfrac{11
     - p}{12} )} \equiv \frac{\Gamma_p (\tfrac{1}{2} )
     \Gamma_p (1 - \tfrac{p}{6} )}{\Gamma_p (\tfrac{7 -
     p}{12} ) \Gamma_p (\tfrac{11 - p}{12} )} \equiv -
     \frac{\Gamma_p (\tfrac{1}{2} )}{\Gamma_p (\tfrac{7}{12}
     ) \Gamma_p (\tfrac{11}{12} )} \quad (\operatorname{mod} p),
  \end{equation*}
  where $\Gamma_p$ is Morita's $p$-adic gamma function (see, for instance,
  \cite[11.6]{cohen-nt2} or \cite[IV.2]{koblitz-p}). Since $\Gamma_p (1 /
  2)^2 = (- 1)^{(p + 1) / 2}$, it follows that, if $p \equiv 1 \pmod{4}$, then
  \begin{equation*}
    C_{\boldsymbol{B}} (\tfrac{p - 1}{2} ) \equiv - \frac{3^{(p - 1) / 2}}{\Gamma_p (\tfrac{7}{12} )^2 \Gamma_p (\tfrac{11}{12} )^2} \equiv - 3^{(p - 1) / 2} \Gamma_p (\tfrac{1}{12} )^2 \Gamma_p (\tfrac{5}{12} )^2 \quad
     (\operatorname{mod} p),
  \end{equation*}
  where we used the $p$-adic version of the reflection formula for the final
  congruence. On the other hand, it follows from the $p$-adic Gauss--Legendre
  multiplication formula (see, for instance, \cite[11.6.14]{cohen-nt2} or
  \cite[p.~91]{koblitz-p}) that, for primes $p \equiv 1 \pmod{4}$,
  \begin{equation*}
    \Gamma_p (\tfrac{1}{12} )^2 \Gamma_p (\tfrac{5}{12}
     )^2 = (\tfrac{3}{p} ) \Gamma_p (\tfrac{1}{4}
     )^4 .
  \end{equation*}
  Since $3^{(p - 1) / 2} \equiv (\frac{3}{p} ) \pmod{p}$, we conclude that, modulo $p$,
  \begin{equation*}
    C_{\boldsymbol{B}} (\tfrac{p - 1}{2} ) \equiv \left\{
     \begin{array}{ll}
       - \Gamma_p (\tfrac{1}{4} )^4, &\quad \text{if $p \equiv 1
       \pmod{4}$,}\\
       0, &\quad \text{if $p \equiv 3 \pmod{4}$.}
     \end{array} \right.
  \end{equation*}
  Comparing with the congruences established for sequence $\boldsymbol{D}$ in
  \cite[(13.4)]{sb-picardfuchs}, we arrive at
  \begin{equation}\label{eq:BC}
    C_{\boldsymbol{B}} (\tfrac{p - 1}{2} ) \equiv C_{\boldsymbol{D}}
     (\tfrac{p - 1}{2} ) \quad (\operatorname{mod} p),
  \end{equation}
  which implies the claimed congruences for sequence $\boldsymbol{B}$.
  \end{proof}

\subsection{Proof of Theorem \ref{thm:C2:L}}\label{sec:L}

This section is devoted to proving Theorem~\ref{thm:C2:L}. We first prove case $\boldsymbol{D}$ in detail, then briefly indicate how to establish cases $\boldsymbol{A}$, $\boldsymbol{B}$, $\boldsymbol{C}$ and $\boldsymbol{E}$.  We conclude with a sketch of case $\boldsymbol{F}$. We note that the relation of the hypergeometric series, which arise for sequences $\boldsymbol{A}$, $\boldsymbol{C}$, $\boldsymbol{D}$, and the corresponding $L$-values already appears in \cite{zudilin-cy-hyp}.

\begin{proof}[Proof of Theorem \ref{thm:C2:L}] We first claim that 
\begin{equation}
    C_{\boldsymbol{D}} (-\tfrac{1}{2}) = \frac{16}{\pi^2} L (f_{\boldsymbol{D}}, 2). \label{eq:apery2:L}
  \end{equation}
Expressing the defining binomial sum hypergeometrically, we have
  \begin{equation*}
    C_{\boldsymbol{D}} (x) =\pFq{3}{2}{- x, - x, x + 1}{1, 1}{1},
  \end{equation*}
  so that, in particular,
  \begin{equation*}
    C_{\boldsymbol{D}} (-\tfrac{1}{2}) =\pFq{3}{2}{\tfrac{1}{2}, \tfrac{1}{2}, \tfrac{1}{2}}{1, 1}{1} .
  \end{equation*}
  We could evaluate the right-hand side using hypergeometric identities (such as, in this case, \cite[Theorem~3.5.5]{aar}).
  Instead, here and in subsequent cases, we find it more fitting to the overall theme to employ modular parametrizations.
  As such, applying Clausen's identity (see, for instance, \cite[Proposition~5.6]{borwein-piagm})
  \begin{equation}
    \pFq{3}{2}{\tfrac{1}{2}, s, 1 - s}{1, 1}{4 x (1 - x)} =\pFq{2}{1}{s, 1 - s}{1}{x}^2, \label{eq:clausen}
  \end{equation}
  with $s = 1 / 2$ and the modular parametrization (see \cite{bb-cubicagm}, \cite[p.63]{zagier-intro123} or
  \cite[(37)]{zagier-de})
  \begin{equation}
    \pFq{2}{1}{\tfrac{1}{2}, \tfrac{1}{2}}{1}{\lambda (\tau)} = \theta_3 (\tau)^2,
    \label{eq:2f1:theta3}
  \end{equation}
  where $\theta_2(\tau) = \sum_{n \in \mathbb{Z} + 1/2} q^{n^2/2}$, $\theta_3(\tau) = \sum_{n \in \mathbb{Z}} q^{n^2/2}$ and $\lambda(\tau) = \Bigl(\frac{\theta_2(\tau)}{\theta_3(\tau)}\Bigr)^4$, we find that
  \begin{equation}
    C_{\boldsymbol{D}} (-\tfrac{1}{2}) =\pFq{2}{1}{\tfrac{1}{2}, \tfrac{1}{2}}{1}{\frac{1}{2}}^2 = \left(\frac{\Gamma^2 (\tfrac{1}{4})}{2 \pi^{3 / 2}} \right)^2, \label{eq:apery2:L:1}
  \end{equation}
 upon taking $\tau = i$, in which case $\lambda (i) = \frac{1}{2}$ and
  $\theta_3 (i)^2 = \frac{\Gamma^2 (1 / 4)}{2 \pi^{3 / 2}}$. On the other
  hand, it is shown by Rogers, Wan and Zucker \cite{rwz-ell} that
  \begin{equation*}
    L (f_{\boldsymbol{D}}, 2) = \frac{\Gamma^4 (\tfrac{1}{4})}{64
     \pi} .
  \end{equation*}
  In light of \eqref{eq:apery2:L:1}, this proves \eqref{eq:apery2:L}. Next, we claim that
\begin{equation}
    C_{\boldsymbol{A}} (-\tfrac{1}{2}) = \frac{8}{\pi^2} L(f_{\boldsymbol{A}},
     2).
     \label{eq:A:L}
  \end{equation}
Proceeding as above, we find that
  \begin{equation*}
    C_{\boldsymbol{A}} (-\tfrac{1}{2}) =\pFq{3}{2}{\tfrac{1}{2}, \tfrac{1}{2}, \tfrac{1}{2}}{1, 1}{- 1} .
  \end{equation*}
  Employing \eqref{eq:clausen} and the modular
  parametrization, we obtain
  \begin{equation}
    C_{\boldsymbol{A}} (-\tfrac{1}{2}) =\pFq{2}{1}{\tfrac{1}{2}, \tfrac{1}{2}}{1}{\tfrac{1 - \sqrt{2}}{2}}^2 = \theta_3 \left(1 + i
    \sqrt{2} \right)^4 = \frac{\Gamma^2 (\tfrac{1}{8}) \Gamma^2
    (\tfrac{3}{8})}{8 \sqrt{2} \pi^3} . \label{eq:CA:L:1}
  \end{equation}
  Again, up to the factor $8 / \pi^2$, this matches the corresponding
  $L$-value evaluation \cite{rwz-ell}
  \begin{equation*}
    L(f_{\boldsymbol{A}}, 2) = \frac{\Gamma^2 (\tfrac{1}{8}) \Gamma^2
     (\tfrac{3}{8})}{64 \sqrt{2} \pi} .
  \end{equation*}
This proves \eqref{eq:A:L}. Now, in order to see  
  \begin{equation}
    C_{\boldsymbol{B}} (-\tfrac{1}{2}) = \frac{8}{\pi^2} L (f_{\boldsymbol{B}}, 2),
    \label{eq:B:L}
  \end{equation}
we begin with
\begin{equation*}
    C_{\boldsymbol{B}} (x) = 3^x \pFq{3}{2}{- \tfrac{x}{3}, - \tfrac{x - 1}{3}, - \tfrac{x - 2}{3}}{1, 1}{1}
  \end{equation*}
  and hence
  \begin{equation*}
    C_{\boldsymbol{B}} (-\tfrac{1}{2}) = 3^{- 1 / 2} \pFq{3}{2}{\tfrac{1}{6}, \tfrac{1}{2}, \tfrac{5}{6}}{1, 1}{1}.
  \end{equation*}
  Let $j(\tau)$ denote Klein's modular function. By \cite[Theorem~5.7]{borwein-piagm}, we have
  \begin{equation*}
    \pFq{3}{2}{\tfrac{1}{6}, \tfrac{1}{2}, \tfrac{5}{6}}{1, 1}{\frac{1728}{j (\tau)}} = \sqrt{1 - \lambda (\tau) (1
     - \lambda (\tau))} \pFq{2}{1}{\tfrac{1}{2}, \tfrac{1}{2}}{1}{\lambda (\tau)}^2,
  \end{equation*}
  which specialized to $\tau = i$, and combined with \eqref{eq:2f1:theta3},
  yields
  \begin{equation*}
    C_{\boldsymbol{B}} (-\tfrac{1}{2}) = 3^{- 1 / 2} \pFq{3}{2}{\tfrac{1}{6}, \tfrac{1}{2}, \tfrac{5}{6}}{1, 1}{1} = \frac{1}{2} \theta_3 (i)^4 = \frac{\Gamma^4
     (\tfrac{1}{4})}{8 \pi^3} .
  \end{equation*}
  Up to the factor $8 / \pi^2$, this equals the $L$-value evaluation
  \cite{rwz-ell}
  \begin{equation*}
    L (f_{\boldsymbol{B}}, 2) = \frac{\Gamma^4 (\tfrac{1}{4})}{64
     \pi} .
  \end{equation*}
Thus, \eqref{eq:B:L} follows. To prove
\begin{equation}
    C_{\boldsymbol{C}} (-\tfrac{1}{2}) = \frac{12}{\pi^2} L (f_{\boldsymbol{C}}, 2),
     \label{eq:C:L}
  \end{equation}
we first observe
\begin{equation*}
    C_{\boldsymbol{C}} (-\tfrac{1}{2}) = \Re \pFq{3}{2}{\tfrac{1}{2}, \tfrac{1}{2}, \tfrac{1}{2}}{1, 1}{4} .
  \end{equation*}
  Employing \eqref{eq:clausen} and the modular
  parametrization, we obtain
  \begin{equation*}
    \pFq{3}{2}{\tfrac{1}{2}, \tfrac{1}{2}, \tfrac{1}{2}}{1, 1}{4} =\pFq{2}{1}{\tfrac{1}{2}, \tfrac{1}{2}}{1}{\tfrac{1 - i \sqrt{3}}{2}}^2 = \theta_3 \left(-
     \tfrac{1 - i \sqrt{3}}{2} \right)^4 = \frac{\left(3 - i \sqrt{3} \right)
     \Gamma^6 (\tfrac{1}{3})}{2^{11 / 3} \pi^4} .
  \end{equation*}
  Up to the factor $12 / \pi^2$, the real part of this equals the $L$-value
  evaluation \cite{rwz-ell}
  \begin{equation*}
    L (f_{\boldsymbol{C}}, 2) = \frac{\Gamma^6 (\tfrac{1}{3})}{2^{17 / 3} \pi^2} .
  \end{equation*}
 This yields \eqref{eq:C:L}. Next, to deduce
  \begin{equation}
    \underset{x = - 1 / 2}{\operatorname{res}} C_{\boldsymbol{E}} (x) = \frac{6}{\pi^2}
     L (f_{\boldsymbol{E}}, 1),
\label{eq:E:L}
  \end{equation}
we start with
\begin{equation*}
    C_{\boldsymbol{E}} (x) = \binom{2 x}{x} \pFq{3}{2}{- x, - x, \tfrac{1}{2}}{\tfrac{1}{2} - x, 1}{- 1}
  \end{equation*}
  and hence
  \begin{equation*}
    \underset{x = - 1 / 2}{\operatorname{res}} C_{\boldsymbol{E}} (x) = \frac{1}{2 \pi}
     \pFq{3}{2}{\tfrac{1}{2}, \tfrac{1}{2}, \tfrac{1}{2}}{1, 1}{- 1} = \frac{\Gamma^2 (\tfrac{1}{8})
     \Gamma^2 (\tfrac{3}{8})}{16 \sqrt{2} \pi^4},
  \end{equation*}
  where the second equality is a consequence of \eqref{eq:CA:L:1}. Up to the
  factor $6 / \pi^2$, this equals
  \begin{equation*}
    L (f_{\boldsymbol{E}}, 1) =
     \frac{\Gamma^2 (\tfrac{1}{8}) \Gamma^2 (\tfrac{3}{8})}{96 \sqrt{2} \pi^2},
  \end{equation*}
  which follows from \eqref{eq:LE:1:2} and the value for $L (f_{\boldsymbol{E}}, 2)$ obtained in
  \cite{rwz-ell}. This proves \eqref{eq:E:L}. Finally, we claim that (see also \cite{watson}, \cite{zucker})
  \begin{equation}
    C_{\boldsymbol{F}}(-\tfrac{1}{2})
    = \frac{6}{\pi^2} L (f_{\boldsymbol{F}}, 2)
    = \frac{\Gamma (\frac{1}{24} ) \Gamma (\frac{5}{24} ) \Gamma (\frac{7}{24} ) \Gamma (\frac{11}{24} )}{16 \sqrt{6} \pi^3}. \label{eq:CF:L}
  \end{equation}
  To begin with, note that
  \begin{equation*}
    C_{\boldsymbol{F}} (- \tfrac{1}{2} ) = \frac{1}{\sqrt{8}}
     \sum_{k = 0}^{\infty} 2^{- 5 k} \binom{2 k}{k} C_{\boldsymbol{A}} (k) =
     \frac{1}{\sqrt{8}} g \left(\frac{1}{32} \right),
  \end{equation*}
  where
  \begin{equation*}
    g (z) = \sum_{k = 0}^{\infty} z^k \binom{2 k}{k} \sum_{j = 0}^k
     \binom{k}{j}^3 .
  \end{equation*}
  In \cite[Theorem~2.1]{ctyz-clausen}, the modular parametrization
  \begin{equation*}
    g \left(\frac{x (\tau)}{(1 - x (\tau))^2} \right) = \frac{1}{6} (6 E_2
     (6 \tau) + 3 E_2 (3 \tau) - 2 E_2 (2 \tau) - E_2 (\tau)),
  \end{equation*}
  with
  \begin{equation*}
    x (\tau) = \left(\frac{\eta (\tau) \eta (6 \tau)}{\eta (2 \tau) \eta (3
     \tau)} \right)^{12}, \quad E_2 (\tau) = 1 - 24 \sum_{n = 1}^{\infty}
     \frac{n q^n}{1 - q^n},
  \end{equation*}
  is obtained. Specializing this parametrization at $\tau = \tau_0 = i /
  \sqrt{6}$, we obtain the desired value $g (1 / 32)$. It then is a standard
  application of the Chowla--Selberg formula \cite{sc67} to show that
  \begin{equation*}
    3 E_2 (3 \tau_0) - E_2 (\tau_0) = 6 E_2 (6 \tau_0) - 2 E_2 (2 \tau_0) =
     \frac{\sqrt{3} \Gamma (\frac{1}{24} ) \Gamma (\frac{5}{24} ) \Gamma (\frac{7}{24} ) \Gamma (\frac{11}{24} )}{8 \pi^3},
  \end{equation*}
  which implies
  \begin{equation}
    C_{\boldsymbol{F}} (- \tfrac{1}{2} )
    = \frac{\Gamma (\frac{1}{24} ) \Gamma (\frac{5}{24} ) \Gamma (\frac{7}{24} ) \Gamma (\frac{11}{24} )}{16 \sqrt{6} \pi^3} . \label{eq:CF:gamma}
  \end{equation}
  That the right-hand side equals the claimed $L$-value then follows from work of Damerell \cite{damerell-l1} because $L(f_{\boldsymbol{F}}, s)$ can be viewed as a Hecke $L$-series on the field $\mathbb{Q}(\sqrt{-6})$ (see also \cite{BFFLM}, \cite{schutt}).
\end{proof}

\subsection{Interpolating the sporadic sequences}\label{sec:interpol}

Zagier's interpolated series \eqref{eq:apery3:x} is absolutely convergent for all $x \in \mathbb{C}$ (as a consequence of \eqref{eq:apery:C:asy}) and defines a holomorphic
function satisfying the symmetry $A (x) = A (- x - 1)$. Zagier shows the
somewhat surprising fact that $A (x)$ does not satisfy the same recurrence as
the Ap\'ery numbers, but instead the inhomogeneous functional
equation
\begin{equation}
  P (x, S_x) A (x) = \frac{8}{\pi^2} (2 x + 3) \sin^2 (\pi x)
  \label{eq:apery:x:fun}
\end{equation}
for all complex $x$, where
\begin{equation}
  P (x, S_x) = (x + 2)^3 S_x^2 - (2 x + 3) (17 x^2 + 51 x + 39) S_x + (x +
  1)^3 \label{eq:apery:x:P}
\end{equation}
is Ap\'ery's recurrence operator, and $S_x$ denotes the (forward) shift operator in $x$, meaning that $S_x f(x) = f(x+1)$.

\begin{remark}
  Let us illustrate how one can algorithmically derive and prove
  \eqref{eq:apery:x:fun}. Let $D (x, k)$ be the summand in the sum defining $A
  (x)$. Creative telescoping applied to $D (x, k)$ determines the operator $P
  (x, S_x)$ given in \eqref{eq:apery:x:P} as well as a rational function $R
  (x, k)$ such that
  \begin{equation}
    P (x, S_x) D (x, k) = (1 - S_k) R (x, k) D (x, k) . \label{eq:apery3:ct}
  \end{equation}
  It follows that
  \begin{equation*}
    P (x, S_x) \sum_{k = 0}^{K - 1} D (x, k) = R (x, 0) D (x, 0) - R (x, K) D
     (x, K) = - R (x, K) D (x, K),
  \end{equation*}
  and it remains to compute the limit of the right-hand side as $K \rightarrow
  \infty$. Using basic properties of the gamma function, as done in
  \cite{zagier-de}, one obtains
  \begin{equation}
    D (x, k) = \left[ \frac{\sin (\pi x)}{\pi k} \right]^2 + O \left(\frac{1}{k^3} \right), \quad k \rightarrow \infty, \label{eq:apery:C:asy}
  \end{equation}
  from which we deduce that $- R (x, K) D (x, K)$ approaches $8 (2 x + 3)
  \sin^2 (\pi x) / \pi^2$ as $K \rightarrow \infty$. The following lines of
  Mathematica code use Koutschan's Mathematica package
  \texttt{HolonomicFunctions} \cite{koutschan-phd} to perform all of these
  computations automatically:
  \begin{alltt}
  Dxk = Binomial[x,k]^2 Binomial[x+k,k]^2
  \{\{P\}, \{R\}\} = CreativeTelescoping[Dxk, S[k]-1, \{S[x]\}]
  \{R\} = OrePolynomialListCoefficients[R]
  Limit[-R Dxk, k->Infinity, Assumptions->Element[k,Integers]]
  \end{alltt}
\end{remark}

\begin{remark}
  We note that the sum in \eqref{eq:apery3:x} actually has natural boundaries,
  meaning that the range of summation can be extended from nonnegative
  integers to all integers without changing the sum. The reason is that the
  summand vanishes for all $x \in \mathbb{C}$ if $k$ is a negative integer.
  More specifically, if $k$ is a negative integer, then $\binom{x + k}{k} = 0$
  for all $x \in \mathbb{C}\backslash \{ - k - 1, - k - 2, \ldots, 1, 0 \}$,
  while $\binom{x}{k} = 0$ for all $x \in \mathbb{C}\backslash \{ - 1, - 2,
  \ldots, k \}$. For more details on binomial cofficients with negative
  integer entries, we refer to \cite{fs-qbinomial} and \cite{loeb-neg}.
\end{remark}

Somewhat unexpectedly, there are marked differences when considering the
interpolations of the sporadic sequences given in Table~\ref{tbl:sporadic2:i}.
For illustration, consider sequence $\boldsymbol{D}$ with interpolation
\begin{equation}
  C_{\boldsymbol{D}} (x) = \sum_{k = 0}^{\infty} \binom{x}{k}^2 \binom{x + k}{k}
  . \label{eq:apery2:x}
\end{equation}
In this case, we find that, as $k \rightarrow \infty$,
\begin{equation*}
  \binom{x}{k}^2 \binom{x + k}{k} \sim \frac{\Gamma (x + 1)}{k^x} \left[
   \frac{\sin (\pi x)}{\pi k} \right]^2,
\end{equation*}
which implies that the series \eqref{eq:apery2:x} converges if $\Re x >
- 1$ but diverges if $\Re x < - 1$. Moreover, proceeding as in the case
of the Ap\'ery numbers $A(n)$, it follows that $C_{\boldsymbol{D}}
(x)$ satisfies the homogeneous functional equation
\begin{equation*}
  [(x + 2)^2 S_x^2 - (11 x^2 + 33 x + 25) S_x - (x + 1)^2] C_{\boldsymbol{D}}
   (x) = 0
\end{equation*}
for all complex $x$ with $\Re x > - 1$. This is recurrence
\eqref{eq:rec2-abc} with $(a, b, c) = (11, 3, - 1)$.

The situation is similar for our interpolations of the sequences
$\boldsymbol{A}$, $\boldsymbol{B}$ and $\boldsymbol{E}$. In each case, the defining
series (see Table~\ref{tbl:sporadic2:i}) converges if $\Re x > - 1$ and
one finds, as in the case of sequence $\boldsymbol{D}$, that the interpolation
satisfies the recurrence \eqref{eq:rec2-abc} for the appropriate choice of
$(a, b, c)$.

\begin{example}
\label{eq:CC:interpol}
  Some care is required for sequence $\boldsymbol{C}$, which has the
  binomial sum representation
  \begin{equation*}
    C_{\boldsymbol{C}} (n) = \sum_{k = 0}^n \binom{n}{k}^2 \binom{2 k}{k} .
  \end{equation*}
  In this case, letting $n$ be a complex variable and extending the summation
  over all nonnegative integers $k$ never yields a convergent sum (unless $n$
  is a nonnegative integer, in which case the sum is a finite one). However,
  the binomial sum can be expressed hypergeometrically as
  \begin{equation}
    C_{\boldsymbol{C}} (n) =\pFq{3}{2}{- n, - n, \tfrac{1}{2}}{1, 1}{4} . \label{eq:CC:hyp}
  \end{equation}
  For integers $n \geq 0$, this hypergeometric series is a finite sum.
  For other values of $n$, we can make sense of the hypergeometric function
  \eqref{eq:CC:hyp} by replacing $4$ with a complex argument $z$ (the series
  converges for $| z | < 1$) and analytic continuation to $z = 4$. As usual,
  the principal branch of the hypergeometric function is chosen by cutting
  from $z = 1$ to $z = \infty$ on the real axis. As a consequence, there is a
  choice to approach $z = 4$ from either above or below the real axis, and the
  two resulting values are complex conjugates of each other. We avoid this
  ambiguity, as well as complex values, by defining
  \begin{equation*}
    C_{\boldsymbol{C}} (x) = \Re \pFq{3}{2}{- x, - x, \tfrac{1}{2}}{1, 1}{4} .
  \end{equation*}
  That this is a sensible choice of interpolation is supported by Theorem~\ref{thm:C2:L}.
\end{example}

\begin{example}
  \label{eg:CF:interpol}For sequence $\boldsymbol{F}$, let us consider the
  interpolation
  \begin{equation}
    C_{\boldsymbol{F}} (x) = \sum_{k = 0}^{\infty} (- 1)^k 8^{x - k}
    \binom{x}{k} C_{\boldsymbol{A}} (k), \label{eq:g:x}
  \end{equation}
  where $C_{\boldsymbol{A}} (n)$ are the Franel numbers
  \begin{equation*}
    C_{\boldsymbol{A}} (n) = \sum_{k = 0}^n \binom{n}{k}^3 = \frac{2
     \sqrt{3}}{\pi}  \frac{2^{3 n}}{3 n} \left(1 + O \left(\frac{1}{n}
     \right) \right) .
  \end{equation*}
  It follows that, as $k \rightarrow \infty$,
  \begin{equation*}
    (- 1)^k 8^{x - k} \binom{x}{k} C_{\boldsymbol{A}} (k) = \frac{2}{\pi
     \sqrt{3}} \frac{8^x}{\Gamma (- x)}  \frac{1}{k^{x + 2}} \left(1 + O
     \left(\frac{1}{k} \right) \right),
  \end{equation*}
  from which we deduce that, once more, the series \eqref{eq:g:x} converges if
  $\Re x > - 1$. Consequently, we expect that the truncation
  \begin{equation*}
    C_{\boldsymbol{F}} (x ; N) = \sum_{k = 0}^N (- 1)^k 8^{x - k} \binom{x}{k}
     C_{\boldsymbol{A}} (k),
  \end{equation*}
  as $N \rightarrow \infty$, has an asymptotic expansion of the form
  \begin{equation*}
    C_{\boldsymbol{F}} (x ; N) = C_{\boldsymbol{F}} (x) + \frac{b_1 (x)}{N^{x +
     1}} + \frac{b_2 (x)}{N^{x + 2}} + \ldots
  \end{equation*}
  Using this assumption, we can speed up the convergence of $C_{\boldsymbol{F}}
  (- 1 / 2 ; N)$ by considering the sequence $c_n = C_{\boldsymbol{F}} (- 1 / 2
  ; n^2)$ and approximating its limit via the differences $(S_n - 1)^m n^m c_n
  / m!$ for suitable choices of $m$ and $n$. This allows us to compute $C_{\boldsymbol{F}}(- 1/2)$
  to, say, 50 decimal places. Namely,
  \begin{equation*}
    C_{\boldsymbol{F}} (-\tfrac{1}{2}) =
     0.50546201971732600605200405322714025998512901481742 \ldots
  \end{equation*}
  This allowed us to numerically discover \eqref{eq:C2:L} for sequence $\boldsymbol{F}$.
  For comparison, summing the first $100, 000$
  terms of the series only produces three correct digits.
\end{example}

\section{Cellular integrals}\label{sec:cellular}

Recently, Brown \cite{brown-apery} introduced a program where period integrals
on the moduli space $\mathcal{M}_{0,N}$ of curves of genus 0 with $N$ marked points
play a central role in understanding irrationality proofs of values of the Riemann zeta function. The idea is to associate a rational function $f_{\sigma}$ and a differential $(N-3)$-form $\omega_{\sigma}$ to a given permutation $\sigma = \sigma_N$ on $\{1, 2, \dotsc, N\}$. Consider the cellular integral 

\begin{equation*}
I_{\sigma}(n) := \int_{S_N} f_{\sigma}^n \,\, \omega_{\sigma},
\end{equation*}
where 
\begin{equation*}
S_{N} = \{(t_1, \dotsc, t_{N-3}) \in \mathbb{R}^{N-3} : 0 < t_1 < \dotsc < t_{N-3} < 1 \}.
\end{equation*}
By \cite[Corollary~8.2]{brown-apery}, $I_{\sigma}(n)$ is a $\mathbb{Q}$-linear combination of multiple zeta values of weight less than or equal to $N-3$. Suppose that this linear combination is of the form $A_{\sigma_N}(n) \zeta_\sigma(N-3)$, with $A_{\sigma_N}(n) \in \mathbb{Q}$, plus a combination of multiple zeta values of weight less than $N-3$. We then say that $A_{\sigma}(n) = A_{\sigma_N}(n)$ is the \emph{leading coefficient} of the cellular integral $I_{\sigma}(n)$. For example, if $N=5$, then $\sigma_5=(1,3,5,2,4)$ is the unique convergent permutation, $I_{\sigma_5}(n)$ recovers Beukers' integral for $\zeta(2)$ \cite{beukers79} and the leading coefficients $A_{\sigma_5}(n)$ are the Ap{\'e}ry numbers $C_{\boldsymbol{D}} (n)$ in \eqref{eq:apery2}.

In \cite{mos-brown}, an explicit family $\sigma_N$ of convergent configurations for odd $N \geq 5$ is constructed such that the leading coefficients $A_{\sigma_N} (n)$ are powers of the Ap\'ery numbers $C_{\boldsymbol{D}} (n)$,
that is,
\begin{equation}
  A_{\sigma_N} (n) = C_{\boldsymbol{D}} (n)^{(N - 3) / 2}. \label{eq:Asigma}
\end{equation}
The first main result in \cite{mos-brown} extends Theorem \ref{thm:C2:mod} for sequence $\boldsymbol{D}$ to a supercongruence for all odd weights greater than or equal to 3. Specifically, for odd $k= N-2 \geq 3$, consider the binary theta series
\begin{equation}
  f_k (\tau) = \frac{1}{4}  \sum_{(n, m) \in \mathbb{Z}^2} (- 1)^{m (k - 1) /
  2} (n - i m)^{k - 1} q^{n^2 + m^2} =: \sum_{n \geq 1} \gamma_k (n) q^n. \label{eq:fk}
\end{equation}

\begin{theorem}[{\cite[Theorem 1.1]{mos-brown}}]
  \label{thm:mos:odd}For each odd integer $N \geq 5$, let $A_{\sigma_N}(n)$ and $f_k (\tau)$ be as in \eqref{eq:Asigma} and \eqref{eq:fk}, respectively. Then, for all primes $p \geq 5$,
  \begin{equation}
    A_{\sigma_N} (\tfrac{p-1}{2}) \equiv \gamma_{k} (p)
    \quad (\operatorname{mod} p^2) . \label{eq:mos:odd:modcong2}
  \end{equation}
\end{theorem}

Using the interpolation \eqref{eq:apery2:x} for $C_{\boldsymbol{D}} (n)$ and \eqref{eq:Asigma}, we have the following analogue of \eqref{eq:apery3:L} for all odd $N \geq 5$.

\begin{theorem}
  \label{thm:mos:odd:L}Let $N \geq 5$ be an odd positive integer, $k=N-2$ and $f_k(\tau)$ be as in \eqref{eq:fk}. Then,
  \begin{equation}
    A_{\sigma_N} (-\tfrac{1}{2}) = \frac{\alpha_k}{\pi^{k - 1}}
    L (f_k, k - 1), \label{eq:mos:odd:L}
  \end{equation}
  where $\alpha_k$ is an explicit rational number given as follows:
  \begin{equation}
    \alpha_k = 2^{(k + 1) / 2} (k - 2) \left\{ \begin{array}{ll}
      2 / r_{(k - 1) / 2}, &\quad \text{if $k \equiv 1 \pmod{4}$},\\
      1 / s_{(k - 1) / 2}, &\quad \text{if $k \equiv 3 \pmod{4}$}.
    \end{array} \right. \label{eq:mos:odd:L:a}
  \end{equation}
  Here, $r_n$ is defined by $r_2 = 1 / 5$, $r_3 = 0$ and 
  \begin{equation}
    (2 n + 1) (n - 3) r_n = 3 \sum_{k = 2}^{n - 2} r_k r_{n - k}
    \label{eq:mos:odd:L:a:rec}
  \end{equation}
 for $n \geq 4$, and $s_n$ is defined by $s_1 = 1 / 4$, $s_2 = 11 / 80$, $s_3 = 1 / 32$ and the same recursion \eqref{eq:mos:odd:L:a:rec} for $n \geq 4$.
\end{theorem}

\begin{proof}[Proof of Theorem \ref{thm:mos:odd:L}]
  Since $f_3 (\tau) = \eta (4 \tau)^6$, the case $N = 5$ is
  \eqref{eq:apery2:L}. Thus, we assume $N > 5$. As a consequence of \eqref{eq:apery2:L:1} and \eqref{eq:Asigma}, we have
  \begin{equation}
    A_{\sigma_N} (-\tfrac{1}{2}) = \left(\frac{\Gamma^2 (\tfrac{1}{4})}{2 \pi^{3 / 2}} \right)^{N - 3} = \left(\frac{\sqrt{2} \omega}{\pi} \right)^{k - 1}, \label{eq:mos:odd:L:w}
  \end{equation}
  where
  \begin{equation*}
    \omega = 2 \int_0^1 \frac{\md x}{\sqrt{1 - x^4}} = \frac{\Gamma^2
     (\tfrac{1}{4})}{2 \sqrt{2 \pi}}
  \end{equation*}
  is the lemniscate constant. On the other hand, it follows from the
  representation \eqref{eq:fk} that
  \begin{eqnarray*}
    L (f_k, k - 1) & = & \frac{1}{4}  \sum_{(n, m) \neq (0, 0)} (- 1)^{m (k -
    1) / 2} \frac{(n - i m)^{k - 1}}{(n^2 + m^2)^{k - 1}}\\
    & = & \frac{1}{4}  \sum_{(n, m) \neq (0, 0)} (- 1)^{m (k - 1) / 2}
    \frac{1}{(n + i m)^{k - 1}} .
  \end{eqnarray*}
  In other words, these $L$-values are values of the Eisenstein series
  \begin{equation*}
    G_{\ell} (\tau) = \sum_{(n, m) \neq (0, 0)} \frac{1}{(n + m \tau)^{\ell}}
  \end{equation*}
  of even weight $\ell$. Specifically, since
  \begin{equation*}
    2 G_{\ell} (2 \tau) - G_{\ell} (\tau) = \sum_{(n, m) \neq (0, 0)}
     \frac{(- 1)^m}{(n + m \tau)^{\ell}},
  \end{equation*}
  we have
  \begin{equation*}
    L (f_k, k - 1) = \frac{1}{4} \left\{ \begin{array}{ll}
       G_{k - 1} (i), & \text{if $k \equiv 1 \pmod{4}$},\\
       2 G_{k - 1} (2 i) - G_{k - 1} (i), & \text{if $k \equiv 3 \pmod{4}$} .
     \end{array} \right.
  \end{equation*}
  We note that, if $k \equiv 3 \pmod{4}$, then $G_{k - 1}
  (i) = 0$ because, writing $k = 4 \ell + 3$,
  \begin{align*}
    G_{k - 1} (i) &= \sum_{(n, m) \neq (0, 0)} \frac{1}{(n + m i)^{4 \ell + 2}} \\
    &= \sum_{(n, m) \neq (0, 0)} \frac{1}{i^{4 \ell + 2} (m - n i)^{4 \ell + 2}} = - G_{k - 1} (i) .
  \end{align*}
  For $n \geq 4$, we have (see \cite[Theorem~1.13]{apostol-mf})
  \begin{equation*}
    (4 n^2 - 1) (n - 3) G_{2 n} = 3 \sum_{k = 2}^{n - 2} (2 k - 1) (2 n - 2 k
     - 1) G_{2 k} G_{2 (n - k)},
  \end{equation*}
  which, upon setting $H_k = (2 k - 1) G_{2 k}$, takes the simplified form
  \begin{equation}
    (2 n + 1) (n - 3) H_n = 3 \sum_{k = 2}^{n - 2} H_k H_{n - k} .
    \label{eq:eisenstein:rec:H}
  \end{equation}
  In terms of the functions $H_k$, we have
  \begin{equation*}
    L (f_k, k - 1) = \frac{1}{4 (k - 2)} \left\{ \begin{array}{ll}
       H_{2 \ell} (i), & \text{if $k = 4 \ell + 1$},\\
       2 H_{2 \ell + 1} (2 i), & \text{if $k = 4 \ell + 3$} .
     \end{array} \right.
  \end{equation*}
  Note that the required values of $H_k (\tau)$ at $\tau = i$ and $\tau = 2 i$
  are determined by the recursive relation \eqref{eq:eisenstein:rec:H} once we
  know the initial cases $k = 2$ and $k = 3$. It is shown, for instance, in
  \cite[Theorem~6]{lmp-hurwitz} that
  \begin{equation*}
    H_2 (i) = 3 G_4 (i) = \frac{\omega^4}{5},
  \end{equation*}
  and our earlier discussion implies $H_3 (i) = 5 G_6 (i) = 0$. Similarly, one
  shows that
  \begin{equation*}
    H_2 (2 i) = 3 G_4 (2 i) = \frac{11 \omega^4}{80}, \quad H_3 (2 i) = 5 G_6
     (2 i) = \frac{\omega^6}{32} .
  \end{equation*}
  In light of these initial values, the recurrence \eqref{eq:eisenstein:rec:H}
  implies that, for $n \geq 2$, the values $H_n (i)$ and $H_n (2 i)$ are
  rational multiples of $\omega^{2 n}$. Moreover, the rational factors are
  given by the sequences $r_n$ and $s_n$:
  \begin{equation*}
    r_n = \frac{H_n (i)}{\omega^{2 n}}, \quad s_n = \frac{H_n (2
     i)}{\omega^{2 n}} .
  \end{equation*}
  Thus,
  \begin{equation*}
    L (f_k, k - 1) = \frac{\omega^{k - 1}}{4 (k - 2)} \left\{
     \begin{array}{ll}
       r_{2 \ell}, & \text{if $k = 4 \ell + 1$},\\
       2 s_{2 \ell + 1}, & \text{if $k = 4 \ell + 3$},
     \end{array} \right.
  \end{equation*}
  and the claim then follows from comparison with \eqref{eq:mos:odd:L:w}.
\end{proof}

\begin{remark}
  Let us indicate that the rational numbers featuring in
  Theorem~\ref{thm:mos:odd:L} are arithmetically interesting in their own
  right, and analogous to Bernoulli numbers.
  The values $G_{4 \ell} (i)$ were first explicitly evaluated by Hurwitz
  \cite{hurwitz-lem} (see \cite{lmp-hurwitz} for a modern account), who
  showed that
  \begin{equation}
    G_{4 \ell} (i) = \sum_{(n, m) \neq (0, 0)} \frac{1}{(n + i m)^{4 \ell}} =
    \frac{(2 \omega)^{4 \ell}}{(4 \ell) !} E_{\ell}, \label{eq:hurwitz}
  \end{equation}
  where the $E_{\ell}$ are positive rational numbers characterized by $E_1 = 1
  / 10$ and the recurrence
  \begin{equation*}
    E_n = \frac{3}{(2 n - 3) (16 n^2 - 1)} \sum_{k = 1}^{n - 1} (4 k - 1) (4
     n - 4 k - 1) \binom{4 n}{4 k} E_k E_{n - k} .
  \end{equation*}
  In terms of the numbers $r_n$ defined in Theorem~\ref{thm:mos:odd:L}, we
  have
  \begin{equation*}
    E_n = \frac{(4 n) !}{2^{4 n}} \frac{r_{2 n}}{4 n - 1} .
  \end{equation*}
  Equation~\eqref{eq:hurwitz}, defining the Hurwitz numbers $E_{\ell}$, can be
  seen as an analog of
  \begin{equation*}
    \sum_{n \neq 0} \frac{1}{n^{2 \ell}} = \frac{(2 \pi)^{2 \ell}}{(2 \ell)
     !} B_{\ell}
  \end{equation*}
  characterizing the Bernoulli numbers $B_{\ell}$. In other words, in the
  theory of Gaussian integers the Hurwitz numbers $E_{\ell}$ play a role
  comparable to that played by the Bernoulli numbers for the usual integers.
  That this analogy extends much further, including to the theorem of von
  Staudt--Clausen, is beautifully demonstrated by Hurwitz
  \cite{hurwitz-lem}.
\end{remark}

\begin{example}
  Let us make the case $N=7$ of Theorem~\ref{thm:mos:odd:L} explicit. The leading coefficients $A_{\sigma_7} (n)$ are the
  squares of the Ap\'ery numbers $C_{\boldsymbol{D}} (n)$ and the modular form
  $f_5 (\tau)$ can alternatively be expressed as
  \begin{equation*}
    f_5 (\tau) = \eta (\tau)^4 \eta (2 \tau)^2 \eta (4 \tau)^4 .
  \end{equation*}
  The Zagier-type $L$-value evaluation proven in Theorem~\ref{thm:mos:odd:L}
  is
  \begin{equation*}
    A_{\sigma_7} (-\tfrac{1}{2}) = \frac{240}{\pi^2} L (f_5, 4).
  \end{equation*}
  It is observed in \cite{rwz-ell} that this and many other $L$-values are
  naturally expressed in terms of integrals of the complete elliptic integral
  $K$; for instance,
  \begin{equation*}
    L (f_5, 4) = \frac{1}{30} \int_0^1 K' (k)^3 \md k = \frac{1}{9}
     \int_0^1 K (k)^3 \md k.
  \end{equation*}
\end{example}

\begin{example}
  The values of the first several $\alpha_k$ in Theorem~\ref{thm:mos:odd:L}
  are $\alpha_3 = 16$, $\alpha_5 = 240$, $\alpha_7 = 2560$, $\alpha_9 = 33600$, $\alpha_{11} = 491520$, $\alpha_{13} = 6864000$ and $\alpha_{15} = \frac{1022361600}{11}$.
\end{example}

Let $L^{\ast} (f, s) = (2 \pi)^{- s} \Gamma (s) L (f, s)$ be the normalized $L$-function for $f$. It follows from the work of Eichler, Shimura and Manin on period polynomials (for example, see \cite{shimura2}) that the critical $L$-values
$L^{\ast} (f, s)$ for odd $s$ (as well as those for even $s$) are algebraic
multiples of each other. Moreover, if $f$ has odd weight $k$, then by virtue
of the functional equation all critical $L$-values $L^{\ast} (f, s)$ are
algebraic multiples of each other. In particular, it follows that
\eqref{eq:mos:odd:L} can be rewritten as
\begin{equation*}
  A_{\sigma_N} (-\tfrac{1}{2}) = \beta_k \frac{L (f_k, 2)}{\pi^2}
\end{equation*}
for some algebraic numbers $\beta_k$. In fact, it appears that the $\beta_k$'s
are rational numbers. 

\begin{example}
  Numerically, the first several values of $\beta_k$ are $\beta_3 =
  16$, $\beta_5 = 48$, $\beta_7 = 4$, $\beta_9 = 14$, $\beta_{11} =
     \frac{1}{33}$, $\beta_{13} = \frac{11}{18}$, $\beta_{15} =
     \frac{1}{33156}$. These values, as well as the relations indicated in Example \ref{eg:mos:fk:lquot}, may in principle be rigorously obtained using, for instance, Rankin's method \cite{shimura1}.
\end{example}

\begin{example}
  \label{eg:mos:fk:lquot}As indicated above, all critical $L$-values $L^{\ast} (f_k, s)$ are algebraic
  multiples of each other. In fact, numerical computations suggest that all
  critical $L$-values are rationally related. The first few cases
  are:
  \begin{eqnarray*}
    L (f_5, 4) & = & \frac{2 \pi}{5} L (f_5, 3) = \frac{\pi^2}{5} L (f_5, 2) =
    \frac{\pi^3}{6} L (f_5, 1),\\
    L (f_7, 6) & = & \frac{3 \pi}{10} L (f_7, 5) = \frac{3 \pi^2}{40} L (f_7,
    4) = \frac{\pi^3}{80} L (f_7, 3) = \frac{\pi^4}{640} L (f_7, 2)\\
    & = & \frac{\pi^5}{3840} L (f_7, 1),\\
    L (f_9, 8) & = & \frac{3 \pi}{10} L (f_9, 7) = \frac{3 \pi^2}{35} L (f_9,
    6) = \frac{4 \pi^3}{175} L (f_9, 5) = \frac{\pi^4}{175} L (f_9, 4)\\
    & = & \frac{\pi^5}{700} L (f_9, 3) = \frac{\pi^6}{2400} L (f_9, 2) =
    \frac{\pi^7}{5040} L (f_9, 1).
  \end{eqnarray*}
We thank Yifan Yang for pointing out that one can prove the relation
\begin{equation*}
L(f_5, 4) = \frac{\pi^2}{5} L(f_5,2)  
\end{equation*}
using Theorem 2.3 in \cite{fy}.
\end{example}

\section{Outlook}\label{sec:outlook}

There are numerous directions for future study. First, motivated by Beukers' and Zagier's numerical investigation of \eqref{eq:rec2-abc}, Almkvist, Zudilin \cite{az-de06} and Cooper \cite{cooper-sporadic} searched for parameters $(a,b,c,d)$ such that the three-term relation
\begin{equation}
(n+1)^3 u_{n+1} = (2n+1)(an^2 + an + b)u_n - n(cn^2 + d)u_{n-1},
\label{eq:recabcd}
\end{equation}
with initial conditions $u_{-1}=0$, $u_0=1$, produces only integer solutions. For $(a,b,c,d) = (17,5,1,0)$, we obtain the Ap{\'e}ry numbers $A(n)$. In total, there are nine sporadic cases for \eqref{eq:recabcd}. It is not currently known if each of these cases has an interpolated version which is related (similar to  \eqref{eq:apery3:L}) to the critical $L$-value of a modular form of weight $4$. Second, we echo the lament in \cite{osz-6f5} concerning the lack of algorithmic approaches in directly proving congruences, such as \eqref{eq:BC}, between binomial sums. Third, can one extend the results in \cite{fy} to verify the cases in Example \ref{eg:mos:fk:lquot} and, more generally, find an explicit formula for the ratio $L(f_k,k-1)/L(f_k,2)$ in terms of a rational number and a power of $\pi$? Fourth, in the context of Section \ref{sec:cellular}, a supercongruence (akin to \eqref{eq:mos:odd:modcong2}) has been proven in \cite{mos-brown} between the leading coefficient 
\begin{equation*}
    A_{\sigma_8} (n) = \sum_{\substack{k_1, k_2, k_3, k_4 = 0 \\ k_1 + k_2 = k_3 + k_4}}^{n} \prod_{i = 1}^4 \binom{n}{k_i} \binom{n + k_i}{k_i} 
    \label{eq:A86}
\end{equation*}
and $\eta(2\tau)^{12}$, the unique newform in $S_6(\Gamma_{0}(4))$. Does there exist a version of Theorem \ref{thm:mos:odd:L} in this case?
Fifth, Zudilin \cite{zudilin-cy-hyp} recently considered periods of certain instances of rigid
Calabi--Yau manifolds, which are expressed in terms of hypergeometric
functions. In these instances, he conjecturally indicated a relation between special bases
of the hypergeometric differential equations and all critical $L$-values of
the corresponding modular forms (these relations include those that we
observed during the proof of Theorem~\ref{thm:C2:L}). From our present
perspective of interpolations of sequences, can one similarly engage all of
the critical $L$-values?
Finally, it would be highly desirable to have a more conceptual understanding of the connection between these (and potentially other) interpolations and $L$-values.

\section*{Acknowledgements} The first author would like to thank the Hausdorff Research Institute for Mathematics in Bonn, Germany for their support as this work began during his stay from January 2--19, 2018 as part of the Trimester Program ``Periods in Number Theory, Algebraic Geometry and Physics". He also thanks Masha Vlasenko for her support and encouragement during the initial stages of this project.
The authors are particularly grateful to Wadim Zudilin for sharing his proof of Theorem~\ref{thm:C2:L} for sequence $\boldsymbol{F}$.

\bibliography{references}
\bibliographystyle{abbrv}

\end{document}